\documentclass[reqno,centertags,11pt]{amsart}
\usepackage{verbatim}
\usepackage{tikz, enumerate, array, amssymb, amsmath}
\usepackage{graphicx, epsfig }
\usepackage{amsthm}
\usepackage{amsbsy}
\usepackage{amsfonts}

\usepackage{amscd}
\usepackage{latexsym}

\usepackage[colorlinks=true]{hyperref}

\usepackage{mathtools}

\newcommand{\C}{\mathbb{C}}
\newcommand{\Z}{\mathbb{Z}}
\newcommand{\R}{\mathbb{R}}

\newcommand{\N}{\mathbb{N}}

\newcommand{\be}{\begin{equation}}
\newcommand {\ee} {\end{equation}}

\newcommand{\f}{\frac}

\newcommand{\ol}{\overline}

\newcommand{\hatt}{\widehat}
\newcommand{\beq}{\begin{equation}}
\newcommand{\eeq}{\end{equation}}
\newcommand{\bdm}{\begin{displaymath}}
\newcommand{\edm}{\end{displaymath}}
\newcommand{\ba}{\begin{align}}
\newcommand{\ea}{\end{align}}
\newcommand{\bpf}{\begin{proof}}
\newcommand{\epf}{\end{proof}}

\newcommand{\supp}{\mathrm{supp}\, }               

\newcommand{\veps}{\varepsilon}

\newcommand{\dav}{{d_{\mathrm{av}}}}

\allowdisplaybreaks

\newcommand{\les}{\lesssim}
\newcommand{\lam}{{\lambda}}

\newcommand{\vp}{{\varphi}}

\newcommand{\de}{{\delta}}
\newcommand{\De}{{\Delta}}

\newcommand{\ka}{{\kappa}}

\newcommand{\utwo}{{U_\Delta^2}}

\newcommand{\up}{{U_\Delta^p}}
\newcommand{\vtwo}{{V_\Delta^2}}

\newcommand{\mrho}[1]{\rho_{\raisebox{-2.3pt}{$\scriptstyle #1$}}}
\newcommand{\tmrho}[1]{\widetilde\rho_{\raisebox{-2.3pt}{$\scriptstyle #1$}}}

\def\normo#1{\left\|#1\right\|}

\def\wt#1{\widetilde{#1}}
\def\wh#1{\widehat{#1}}

\newtheorem{theorem}{Theorem}
\newtheorem{proposition}[theorem]{Proposition}
\newtheorem{lemma}[theorem]{Lemma}

\theoremstyle{definition}

\newtheorem{remark}[theorem]{Remark}

\newcounter{theoremi}[theorem]

\numberwithin{theorem}{section}
\numberwithin{equation}{section}

\newcounter{smalllist}

\newcounter{listi}
\newenvironment{theoremlist}{\begin{list}{{\rm(\roman{listi})}}{%
\setlength{\topsep}{0mm}\setlength{\parsep}{0mm}\setlength{\itemsep}{0mm}%
\setlength{\labelwidth}{1.5em}\setlength{\leftmargin}{1.7em}\usecounter{listi}%
}}{\end{list}}

\newcounter{smallenum}

\newcounter{assumptions}

\begin{document}
\title[Scattering for dispersion managed NLS]{Scattering for the dispersion managed nonlinear Schr\"odinger equation}
\author{Mi-Ran Choi, Kiyeon Lee, Young-Ran Lee}

\address{Department of Mathematics, Sogang University, 35 Baekbeom--ro,
  Mapo--gu, Seoul 04107, South Korea.}%
\email{mrchoi@sogang.ac.kr}

\address{Stochastic Analysis and Application Research Center(SAARC), Korea Advanced Institute of Science and Technology, 291 Daehak-ro, Yuseong-gu, Daejeon 34141, South Korea.}
\email{kiyeonlee@kaist.ac.kr}

\address{Department of Mathematics, Sogang University, 35 Baekbeom--ro,
  Mapo--gu, Seoul 04107, South Korea.}%
\email{younglee@sogang.ac.kr}
\begin{abstract}
We consider the dispersion managed nonlinear Schr\"dinger equations with quintic and cubic nonlinearities in one and two dimensions, respectively. We prove the global well-posedness and scattering in $L_x^2$ for small initial data employing the $U^p$ and $V^p$ spaces.
\end{abstract}

\date{\today}
\maketitle

\section{Introduction}\label{introduction}

We consider the Cauchy problem for the dispersion managed nonlinear Schr\"odinger equation (NLS) in one and two dimensions
\beq\label{eq:main}
\begin{cases}
   \displaystyle{i\partial_t u+ \dav \Delta u+ \int_0^1 e^{-ir\Delta}(|e^{ir\Delta} u|^{2\sigma}e^{ir\Delta} u)dr=0},\\
   u(0,\cdot)=u_0 ,
\end{cases}
\eeq
where $u=u(t,x):\R\times \R^d \to \C$, $d=1,2$, $\dav\in \R$, $\sigma >0$, and $e^{ir\De}$ is the linear propagator of the  Schr\"odinger equation.
The equation in \eqref{eq:main} is related to the NLS with a periodically varying coefficient
\beq\label{eq:original}
i\partial_t u +d(t)\Delta u+|u|^{2\sigma}u=0.
\eeq
In the regime of strong dispersion management, the dispersion $d(t)$ is a periodic function given by
\[
d(t)=\dav +\f{1}{\veps}d_0\left(\f{t}{\veps}\right),
\]
where $d_0$ is its mean zero part over one period, $\dav$ the average component, and $\veps$ a small positive parameter. Through an appropriate change of variables and averaging over one period, the resulting equation takes the form of \eqref{eq:main} in the model case, that is, $d_0$ is the $2$-periodic function with $d_0(t)=1$ on $[0,1)$ and $d_0(t)=-1$ on $[1,2)$, see \cite{GT1, GT2}. The validity of this averaging process is confirmed in \cite{CKL, CLA, ZGJT01}.

Equation \eqref{eq:original} arises in various physical scenarios, such as the propagation of electromagnetic pulses in optical fiber communication, beam propagation in waveguide arrays, and the study of nonlinear matter waves of Bose-Einstein condensates, see \cite{Abdullaev, Konotop, Serkin}. The balance between dispersion and nonlinearity is crucial for the existence of stable pulses. Dispersion management has proven incredibly successful in producing stable soliton-like pulses, see the review \cite{review1, review2}. Numerous well-known results regarding ground states, which generate solitons of equation \eqref{eq:main}, can be found in the literature, refer to \cite{CHLT, CHL, ZGJT01} for example.

The one-dimensional Cauchy problem \eqref{eq:main} is globally well-posed in $H^1(\R)$ for all $\sigma>0$ when $\dav <0$; for all $0<\sigma <4$ when $\dav>0$;  for all $0<\sigma <2$ in the singular case, i.e., when $\dav=0$, in \cite{AK, CHL, ZGJT01}. Recently, the existence of finite time blowup solutions has been shown in \cite{CHgL} for all $\sigma \ge 4$ when $\dav>0$.
We expect that analogous results can be shown for the two-dimensional problem, however, as of now, there are no such results.

We are interested in analyzing  the long-time behavior of the solution for \eqref{eq:main} when $\dav \neq 0$.
In \cite{MurphyHoose}, the authors established the optimal decay bound and \emph{modified} scattering of the solution for \eqref{eq:main} with small initial data in an appropriate weighted function space for the Kerr nonlinearity, $\sigma=1$, in one dimension. They proved this by adapting the Kato-Pusateri method from \cite{kapu2011}. However, to the best of our knowledge, there are no results available in the literature about the \emph{linear} scattering phenomenon of the solution and this is our main motivation for the present study. Now we present our main result on the global well-posedness and {\it linear} scattering in $L_x^2(\R^d)$ for small initial data.

\begin{theorem}\label{thm:scattering}
Assume $\dav\neq0$ and $\sigma =2/d$ for $d=1,2$.
Given initial data $u_0\in L_x^2(\R^d)$ with small enough $L_x^2$-norm, the Cauchy problem \eqref{eq:main} is globally well-posed in $L_x^2(\R^d)$. Moreover, the solution $u(t)$ scatters in $L_x^2(\R^d)$, that is,
 there exist $\vp_{\pm} \in L_x^2(\R^d)$ such that
\begin{equation}\label{def:scattering}
	\lim_{t\to\infty}\normo{u(t) - e^{\dav i t\De}\vp_+}_{L_x^2} = 0 \quad\mbox{and}\quad
    \lim_{t\to-\infty}\normo{u(t) - e^{\dav i t\De}\vp_-}_{L_x^2}=0. \notag
\end{equation}
\end{theorem}

\smallskip

\begin{remark}
The global well-posedness and scattering  in $H^s(\R^d)$, $s > 0$, can analogously be shown for small initial data, see Remark \ref{rem:general-regul}.
\end{remark}

To prove this theorem, as usual, we consider the Duhamel formula
\begin{align}\label{eq:duhamel}
u(t) = e^{\dav it\Delta}u_0 + i \int_0^t e^{ \dav  i (t-s)\Delta} \int_0^1 e^{-ir\Delta}\left(|e^{ir\Delta} u|^{4/d} e^{ir\Delta} u \right)\!(s)\,dr ds.
\end{align}
The global well-posedness of \eqref{eq:duhamel} holds for {\it every} initial data in $L_x^2(\R^d)$ without imposing any restriction on the size of the initial data. This can be shown by applying the contraction mapping argument combined with mass conservation, relying on the estimates  in $L_x^2(\R^d)$ for the nonlinear part of \eqref{eq:duhamel}, see, e.g., \cite[Lemma 2.5]{CHL} for the one dimensional case. Such an approach is standard for proving the well-posedness of the classical NLS as well as the dispersion managed NLS.

However, it is not possible to prove the scattering result for the Cauchy problem \eqref{eq:main} via the well-known argument for the classical NLS, based on Strichartz's estimates.
This limitation arises from  the effect of the linear propagator $e^{-ir\Delta}$ in nonlinearity, which restricts the choice of the Strichartz admissible pairs for the entire time interval. Consequently, we are unable to use the contraction mapping argument for the suitable space-time norms to obtain the scattering result.
Fortunately, we managed to overcome this difficulty by deriving bilinear estimates in Proposition \ref{prop:bilinear-up} for functions in the adapted $U^p$ and $V^p$ spaces when $\sigma=2/d$. The $U^p$ and $V^p$ spaces were first introduced by H. Koch with D. Tataru \cite{kota2005, kota2007} and  N. Wiener \cite{Wiener}, respectively. This scattering in $L_x^2(\R^d)$ for small initial data, stated in Theorem \ref{thm:scattering}, is the first result on the linear scattering for the dispersion managed NLS. Furthermore, our approach establishes well-posedness and scattering for \eqref{eq:main} with $\sigma= 2/d +k$ in $H^{s}(\R^d)$, where $k \in \N$ and $s=d/2-1/\sigma$.

\bigskip

The plan for the paper is as follows:
In the next section, we introduce function spaces and their properties with the notations which are used in the paper frequently. In Section \ref{sec:bilinear}, we prove bilinear estimates which are main ingredients for proving our main result. We close the paper by giving the proof of Theorem \ref{thm:scattering} via contraction mapping argument in Section \ref{sec:mainproof}.


\section{Preliminaries}\label{sec:pre}
\subsection{Notations}

We use the convention $f\lesssim g$ or $g\gtrsim f$, if there exists a finite constant $C>0$ such that $f\leq Cg$. If $f\lesssim g$ and $f\gtrsim g$, we write  $f \approx g$.

We denote by $ L_t^q(I; L_x^r(\R^d))$, for $1\leq q,r <\infty$ and an interval $I\subset \R$, the space of all functions $u$ for which
\[
\|u\|_{L_t^q(I; L_x^r(\R^d))}= \left(\int_I \left(\int_{\R^d}|u(t,x)|^{r}dx\right)^{\f{q}{r}}dt\right)^{\f{1}{q}}
\]
is finite. If $q=\infty$ or $r=\infty$, use the essential supremum instead, respectively.
 For notational simplicity, we write $L_t^qL_x^r$ for $L_t^q(I; L_x^r(\R^d))$ and $L^q_{t,x}$ or $L_{t,x}^q(I\times \R^d)$ for $L_t^q(I; L_x^q(\R^d))$.
For $1\leq p\leq \infty$, we define the dual exponent $1\leq p'\leq \infty$ by
\[
\f{1}{p}+\f{1}{p'}=1.
\]

The Fourier transform is defined by
$$
\widehat{f}(\xi) = \mathcal{F}(f)(\xi) = \int_{\mathbb{R}^d} e^{- ix\cdot \xi} f(x)\,dx
$$
and the space-time Fourier transform is defined by 
$$
\mathcal{F}_{t,x}(F)(\tau, \xi) = \int_{\R} \int_{\mathbb{R}^{d}} e^{- it \tau}e^{- ix\cdot \xi} F(t,x)\,dx dt.
$$

Let $\rho$ be a Littlewood-Paley function, i.e., a smooth cut-off function satisfying $\rho(\xi) = 1$  for $|\xi|\le 1$ and supported in $|\xi|\le 2$.
Denote by $2^\Z$ the set of dyadic numbers $N=2^j$ where $j\in \Z$.
We define  $\rho_1(\xi):= \rho(\xi)$ and $\mrho{N}(\xi):= \rho (\xi/N)-\rho(2\xi/N)$ for a dyadic number $N \ge 2$.
Note that $\sum_{N \ge 1} \mrho{N}(\xi)=1$.
Then the frequency projection $P_N$ is defined by
\[
\mathcal{F}(P_N f)(\xi) := \mrho{N}(\xi)\widehat{f}(\xi).
\]
For any $N_1, N_2 \in 2^\Z$, $N_1 \sim N_2$ means $2^{-10}N_1 < N_2 < 2^{10}N_1$ while $N_1 \ll N_2$ or $N_2 \gg N_1$ means $2^{10}N_1 \le N_2$. We also denote  $\max(N_1,N_2)$ and $\min(N_1,N_2)$ by $N_{12}^{\max}$ and $N_{12}^{\min}$, respectively.\\

\subsection{$U^p-V^p$ and adapted function spaces}\label{sec:pre2}
As we mentioned above, we introduce $U^p$ and $V^p$ spaces and discuss their properties. These function spaces were introduced by Koch-Tataru \cite{kota2005,kota2007} for NLS.
Later these spaces have been exploited to show the low regularity well-posedness and scattering results of several dispersive equations and related ones, as a useful replacement of $X^{s,b}$-spaces in the limiting cases. We refer to \cite{haheko2009, haheko2010-2, kotavi-book} for these properties and explicit definitions for general dispersive equations.

To state the definitions we need, let $\mathcal{P}$ be the collection of finite partitions $\{ t_0 , \cdots, t_K\}$ satisfying
\[
-\infty < t_0 < \cdots < t_K \le \infty.
\]
If $t_K =\infty$, by convention, $u(t_K) := 0$ for any $u : \mathbb R \to L_x^2(\mathbb R^d)$.
For $1 \le p < \infty$, let us define a $U^p$-atom by a step function $a : \mathbb{R} \to L_x^2(\mathbb R^d)$ of the form
$$
a(t) = \sum_{k=1}^K\phi_k \chi_{[t_{k-1},t_k)}(t) \quad \mbox{ with } \quad\sum_{k=1}^K\|\phi_k\|_{L_x^2}^p=1.
$$
Then the $U^p$ space is defined by
$$
U^p := \Biggl\{ u = \sum_{j=1}^\infty \lam_j a_j : a_j \mbox{ are } U^p\mbox{-atoms, }\lam_j \in \C \mbox{ with }\sum_{j=1}^\infty|\lam_j|<\infty \Biggr\}
$$
with the equipped norm
$$
\|u\|_{U^p}:= \inf \left\{ \sum_{j=1}^\infty|\lam_j| : u = \sum_{j=1}^\infty \lam_j a_j,\ a_j \mbox{ are } U^p\mbox{-atoms, }\lam_j \in \C\right\}.
$$
We next define $V^p$ as the space of all right-continuous functions $v : \mathbb{R} \to L_x^2$ satisfying $v(t) \to 0$ as $t \to -\infty$ and the norm
$$
\|v\|_{V^p} := \sup_{\{t_k\} \in \mathcal{P}} \left( \sum_{k=1}^K\|v(t_k)  - v(t_{k-1})\|_{L_x^2}^p \right)^{\frac1p}
$$
is finite.


We gather the properties of the $U^p$ and $V^p$ spaces we need, including embeddings, duality, and logarithmic interpolation. We begin with the embedding property which follows from Proposition 2.2 and Corollary 2.6 in \cite{haheko2009}. To obtain this embedding property, the conditions that $v$ is right-continuous and that $v(t) \to 0$ as $t \to -\infty$ in the definition of $V^p$ space are necessary, see \cite{haheko2009} for more details.
\begin{lemma}\label{lem:embedd}
Let $1 \le p < q <\infty$, then
	\begin{theoremlist}
		\item $U^p$ and $V^p$ are Banach spaces.
		\item The embeddings $U^p \hookrightarrow V^p \hookrightarrow U^q \hookrightarrow L^{\infty}(\mathbb{R};L_x^2) $ are continuous.
	\end{theoremlist}	
\end{lemma}

\begin{lemma}[Corollary 4.24. of \cite{kotavi-book}]\label{lem:duality} For $1 < p < \infty$, if $u \in U^p$ is absolutely continuous, then
	$$
	\|u\|_{U^p} = \sup \left\{ \left| \int \left<u'(t), v(t) \right>_{L_x^2} dt \right| : v \in C_0^{\infty}, \;\; \|v\|_{V^{p'}}=1 \right\}.
	$$
\end{lemma}

\begin{lemma}\label{lem:limit}
	For $1\le p <\infty$, if $v\in V^p$, then the limits  of  $v(t)$ in $L_x^2(\R^d)$ exist as $t\to \pm\infty$.
\end{lemma}
\begin{proof}
	Without loss of generality, we consider the limit of $v(t)$ as $t\to \infty$ when $\|v\|_{V^p}=1$. Given $0<\veps <1$, there exists a partition $\{t_k\}_{k=1}^K\in \mathcal{P}$ with $t_K<\infty$ such that
	\[
	\big(\|v\|_{V^p}-\veps\big)^{\f{1}{p}}< \left(\sum _{k=1}^K \|v(t_k)-v(t_{k-1})\|^p_{L_x^2}\right)^{\f{1}{p}} \leq \|v\|_{V^p}=1
	\]
	which gives
	\[
	\|v\|_{V^p} < \sum _{k=1}^K \|v(t_k)-v(t_{k-1})\|^p_{L_x^2}+\veps.
	\]
	Then for any $t, t'$ with $t>t'>t_K$, we have $\{t_k\}_{k=1}^K \cap \{ t',t\} \in \mathcal P$ and this implies that
	\[
	\begin{aligned}
		& \sum _{k=1}^K \|v(t_k)-v(t_{k-1})\|^p_{L_x^2} + \|v(t')-v(t_K)\|^p_{L_x^2} + \|v(t)-v(t')\|^p_{L_x^2}  \\
		&\leq  \|v\|_{V^p}\leq
		\sum _{k=1}^K \|v(t_k)-v(t_{k-1})\|^p_{L_x^2}+\veps,
	\end{aligned}
	\]
	and therefore $\|v(t)-v(t')\|_{L_x^2}\leq \veps^{\f{1}{p}}$.
\end{proof}

Motivated by the Duhamel formula \eqref{eq:duhamel}, we define the adapted function spaces $U_{\De}^p$ and $V_{\De}^p$ for $1\le p <\infty$ by
$$
U_{\Delta}^p : = \left\{ e^{\dav i t\Delta}u :  u \in U^p \right\} \;\;\mbox{and}\;\;
V_{\Delta}^p : = \left\{e^{\dav i t \Delta} v :  v \in V^p \right\}
$$
with norms
$$
\|u\|_{U_{\Delta}^p} : =\|e^{- \dav i t \Delta} u\|_{U^p} \;\;\mbox{and}\;\;  \|v\|_{V_{\Delta}^p} : =\|e^{- \dav i t\Delta} v\|_{V^p}.
$$
Then, it follows from Lemma \ref{lem:embedd} that $U_{\De}^p$ and $V_{\De}^p$ are Banach spaces and that the embedding property holds, i.e.,  $U_{\De}^p \hookrightarrow V_{\De}^p \hookrightarrow U_{\De}^q \hookrightarrow L^{\infty}(\mathbb{R};L_x^2)$ are continuous for all $1 \le p < q < \infty$.
Moreover, it is obvious to see that the $V^p$ norm is preserved under the unitary operator $e^{ir\Delta}$ for any $r\in\R$ in $L_x^2(\R^d)$ and so is the $V_\De^p$ norm, that is,
\begin{equation}\label{eq:persistence-vp}
	\left\|e^{ir\Delta} v\right\|_{V_\De^p} = \|v\|_{V_\De^p}
\end{equation}
for all $v\in V_\De^p$ and $r\in\R$. The $U_\De^p$ norm is, also, preserved under $e^{ir\Delta}$, i.e.,
\begin{equation}\label{eq:persistence-up}
	\|e^{ir\Delta} u\|_{U_\De^p} = \|u\|_{U_\De^p}
\end{equation}
for all $u \in \up$ and $r\in\R$. Indeed, for any $U^p$-atom $a$, it is clear that $e^{ir\Delta}a$ is a $U^p$-atom since $e^{ir\Delta}$ is a unitary operator in $L_x^2(\R^d)$. Thus, for all $r\in\R$,  $\|e^{ir\Delta} u\|_{U^p} = \|u\|_{U^p}$ and therefore $\|e^{ir\Delta} u\|_{U_\De^p} = \|u\|_{U_\De^p}$.
We close this section by presenting two lemmas adapted from \cite[Lemma 4.12]{kotavi-book} and \cite[Proposition 2.19]{haheko2009}.
\begin{lemma}[Logarithmic interpolation]\label{lem:log-inter} Let $p > 2$. Then there is a constant $\kappa=\kappa(p) > 0$ such that for all $v \in \vtwo$ and  $M \ge 1$ there exist $u \in \utwo$ and $w \in U_{\De}^p$ with
	$$
	v = u +w
	$$
	for which
	$$
	\frac{\kappa}{M} \|u\|_{\utwo} + e^M\|w\|_{U_{\De}^p} \le \|v\|_{\vtwo}.
	$$
\end{lemma}

\begin{lemma}[Transference principle]\label{lem:transfer}
 If a $k$-linear operator
	$
	T_0 : L_x^2(\R^d) \times \cdots \times L_x^2(\R^d) \to L_{loc}^1(\R^d)
	$
	satisfies
	$$
	\left\|T_0\left( e^{\dav it\De}f_1 , e^{\dav it\De}f_2 , \cdots ,e^{\dav  it\De}f_k \right)\right\|_{L_t^qL_x^r} \les \prod_{j=1}^{k}\|f_j\|_{L_x^2}
	$$
	for some $1 \le q,r \le \infty$, then there exists $T : U^p_\De \times \cdots \times U^p_\De \to L_t^qL_x^r$ satisfying
	$$
	\|T(u_1, u_2,\cdots, u_k)\|_{L_t^q L_x^r} \les \prod_{j=1}^k \|u_j\|_{U_{\De}^q}
	$$
such that
	$T(u_1,\cdots, u_k)(t,\cdot)=T_0(u_1(t),\cdots,u_k(t))$ a.e. in $\R^d$.
\end{lemma}

\section{Bilinear estimates}\label{sec:bilinear}
In this section, we provide the essential tools for establishing the proof of our main theorem.
We begin with invoking the standard Strichartz estimate, well-known in arbitrary dimensions. However, we need its one and two dimensional versions only, as stated below.

\begin{lemma}[Strichartz estimate]\label{lem:stri}
	If $(q,r)$ satisfies $\frac2q   = d \left(\frac12 - \frac1r \right)$ and $(d,q,r) \neq (2,2,\infty)$, then
	\begin{align}\label{eq:stri}
		\|e^{it\Delta} f\|_{L_t^q L_x^r} \les \|f\|_{L_x^2}
	\end{align}
for all $f \in L_x^2(\R^d)$.
\end{lemma}
To estimate the nonlinear part of Duhamel's formula \eqref{eq:duhamel}, we need the following bilinear Strichartz estimate
	\begin{equation}\label{eq:bilinear1}
		\|e^{ \dav it\Delta}f\, e^{\dav it \Delta}  g\|_{L_{t,x}^{1+2/d}} \les  \|f\|_{L_x^2}\| g\|_{L_x^2}
	\end{equation}
for all $f,g \in L_x^2(\R^d)$, which follows from the Cauchy-Schwarz inequality and the Strichartz estimate with $q=r$.
To account for high-low frequency interactions, we establish additional bilinear estimate.
Its proof is quite standard, but we provide the proof for the reader's convenience. The bilinear estimate in $L^2_{t,x}(\R \times \R^2)$ is well-known, as seen in \cite{bourgain}, and an analogous bilinear estimate in $L^2_{t,x}(\R \times \R^d)$ holds for any dimension, see, e.g., \cite[Theorem 2.9]{kotavi-book}.

\begin{lemma}[Bilinear Strichartz estimate]\label{lem:bilinear}
	If  $N_1 \nsim N_2$, then
	\begin{equation}\label{eq:bilinear}
		\normo{e^{\dav it \Delta} P_{N_1}f \, e^{\dav it \Delta} P_{N_2} g}_{L_{t,x}^{1+2/d}} \les \left( \frac{N_{12}^{\min}}{N_{12}^{\max}}\right)^{\beta(d)}  \| P_{N_1}f\|_{L_x^2}\| P_{N_2}g\|_{L_x^2}
	\end{equation}
for all $f, g \in L_x^2(\R^d)$, where $\beta(1)=1/6$ and $\beta(2)=1/2$.
\end{lemma}

\begin{proof}
Without loss of generality, we assume that $N_1 \gg N_2$. Note that the left side of \eqref{eq:bilinear} does not change when we replace $e^{\dav it \Delta} P_{N_2} g$ by its complex conjugate. First, we show that the space-time Fourier transform is
	\begin{align}\label{eq:fourier-side}
		\begin{aligned}
			&\mathcal F_{t,x}\left[e^{\dav it \Delta}f \, \overline{e^{\dav it \Delta}  g}  \right] (\tau,\xi) \\
			&= \frac{1}{(2\pi)^{d-1}} \int_{\R^d} \wh{f}(\xi + \eta) \;\overline{\wh{g}(\eta)}\; \de(\tau + \dav|\xi+\eta|^2 - \dav|\eta|^2)   d\eta.
		\end{aligned}
	\end{align}
 Indeed, using the formula of the linear propagator
	\[
	\left(e^{\dav it \Delta}\vp\right)(x)=\frac{1}{(2\pi)^d}\int_{\R^d} e^{ix\cdot \xi}e^{-\dav it|\xi|^2}\hatt{\vp}(\xi)d\xi,
	\]
	we first see that the space Fourier transform as
 \[
 \begin{aligned}
  & \mathcal F\left[e^{\dav it \Delta}f \, \overline{e^{\dav it \Delta}  g}  \right] (\xi)\\
  &=\frac{1}{(2\pi)^{2d}}  \int_{\R^{2d}}\left(\int_{\R^d} e^{-ix\cdot (\xi-\xi_1+\xi_2)} dx\right) \, e^{-\dav it\left(|\xi_1|^2-|\xi_2|^2\right)}\hatt{f}(\xi_1) \overline{\hatt{g}(\xi_2) }d\xi_1  d\xi_2 \\
  &= \frac{1}{(2\pi)^{d}}  \int_{\R^{2d}} \delta(\xi-\xi_1+\xi_2 )e^{-\dav it\left(|\xi_1|^2-|\xi_2|^2\right)}\hatt{f}(\xi_1) \overline{\hatt{g}(\xi_2) }d\xi_1  d\xi_2 \\
   &= \frac{1}{(2\pi)^{d}}  \int_{\R^{d}} e^{-\dav it\left(|\xi+\xi_2|^2-|\xi_2|^2\right)}\hatt{f}(\xi+\xi_2) \overline{\hatt{g}(\xi_2) } d\xi_2
 \end{aligned}
 \]
 where we used the Dirac delta function  $\delta(\xi)=\f{1}{(2\pi)^d} \int_{\R^d}e^{-i x \cdot \xi}dx$ as distribution.  Subsequently, applying the Fourier transform in time yields \eqref{eq:fourier-side} in a similar fashion.

\medskip
	Now we consider the one dimensional case. Since the Dirac delta function $\delta$ is homogeneous of degree $-1$, that is,
 \[
 \int_\R \vp(x)\delta(ax)dx =\int_\R \vp(y)\f{1}{|a|}\delta(y)dy,
 \]
 it follows from \eqref{eq:fourier-side}
	that 		\begin{align*}
		\mathcal F_{t,x}\left[e^{\dav it \Delta}f\, \overline{e^{\dav it \Delta}  g}  \right] (\tau,\xi)  =\frac1{2|\dav||\xi_1-\xi_2|} \wh{f}(\xi_1) \overline{\wh{g}(\xi_2)}
	\end{align*}
	where $(\xi_1,\xi_2)$ is the solution of
	\begin{align*}
		\dav (\xi_2^2 -  \xi_1^2) = \tau  \mbox{ and } \xi_1-\xi_2 = \xi.
	\end{align*}
	Thus, using the Hausdorff-Young inequality, we see that $|\xi_2-\xi_1|=|\xi| \gtrsim N_1$, and
	\[
	d\tau d\xi = 2|\dav| |\xi_1 - \xi_2|  d\xi_1 d\xi_2,
	\]
	we obtain
	\begin{align*}
&\normo{e^{\dav it  \Delta} P_{N_1}f\,\overline{e^{\dav it  \Delta} P_{N_2} g} }_{L_{t,x}^3}\\
	&	\les \, \normo{\mathcal F_{t,x}\left[e^{\dav it  \Delta} P_{N_1}f \,\overline{e^{\dav it  \Delta} P_{N_2} g}  \right] }_{L_{\tau,\xi}^\frac32}\\
		&\les \, \left( \iint_{\R\times \R} \frac{1}{|\xi_1-\xi_2|^{3/2}}\left|\wh{P_{N_1}f}(\xi_1)\right|^\frac32
		\left|\wh{P_{N_2}g}(\xi_2)\right|^\frac32 |\xi_1-\xi_2| d\xi_1 d\xi_2 \right)^\frac23\\
		&\les \, \frac{1}{N_1^{1/3}} \|\wh{P_{N_1}f}\|_{L_\xi^\frac32} \|\wh{P_{N_2}g}\|_{L_\xi^\frac32} \\
	&	\les \, \left(\frac{N_2}{N_1}\right)^\frac16 \|P_{N_1}f\|_{L_x^2} \|P_{N_2}g\|_{L_x^2}.
	\end{align*}
	The last inequality holds since
	\[
	\|\wh{P_{N}f}\|_{L_\xi^\frac32}^{\frac{3}{2}}  = \int_\R |\mrho{N} (\xi)|^{\frac32} | \wh{f} (\xi)|^{\frac32} d\xi
	\le \bigl(\mbox{supp} (\mrho{N})\bigr)^{\frac14} \left(\int_\R |\mrho{N} (\xi)|^2| \wh{f} (\xi)|^2 d\xi\right)^{\frac34}
	\les N^{\frac{1}{4}}\|P_Nf\|_{L_x^2}^\frac{3}{2}
	\]
	by H\"older's inequality.
	
	\medskip
	Let us move on to the two dimensional case. By the Plancherel identity and \eqref{eq:fourier-side}, we have
	\begin{align*}
		&\normo{e^{\dav it\De} P_{N_1}f \,\ol{e^{\dav it\De}P_{N_2}g} }_{L_{t,x}^2}^2\\
		&\les \normo{\mathcal F_{t,x}\left[e^{\dav it \Delta} P_{N_1}f \,\overline{e^{\dav it \Delta} P_{N_2} g}  \right]  }_{L_{\tau,\xi}^2}^2\\
		& \les \int_\R\int_{\R^2} \left|\int_{\R^2} \mrho{N_1}(\xi + \eta)\wh{f}(\xi + \eta) \;\mrho{N_2}(\eta)\overline{\wh{g}(\eta)}\; \de(\tau + \dav|\xi+\eta|^2 - \dav|\eta|^2)   d\eta \right|^2 d\xi d\tau.
	\end{align*}
	Setting $\tmrho{1}=\mrho{1}+\mrho{2}$ and
	$\tmrho{N}= \mrho{ N/2}+ \mrho{ N}+\mrho{ 2 N}$ for $N \ge 2$, then note $\mrho{N}=\mrho{N} \tmrho{N}$ and $(\tmrho{N})^2 \le \tmrho{N}$.
    Thus, using the Cauchy-Schwarz inequality in $\eta$ and the Plancherel identity, we obtain
	\begin{align*}
		&\normo{e^{\dav it\De} P_{N_1}f \,\overline{e^{\dav it\De}P_{N_2}g}}_{L_{t,x}^2}^2\\
  &\les  \normo{\mathcal M}_{L_{\tau,\xi}^\infty}\int_\R \int\!\!\!\!\int_{\R^2\times \R^2}  | \mrho{N_1}(\xi+\eta)\wh{f}(\xi+\eta) |^2 | \mrho{N_2}(\eta)\wh{g}(\eta)|^2
 \, \de(\tau +\dav |\xi+\eta|^2 -\dav |\eta|^2) d\eta d\xi d\tau\\
		&\les\normo{\mathcal M}_{L_{\tau,\xi}^\infty}\|P_{N_1}f\|_{L_x^2}^2\|P_{N_2}g\|_{L_x^2}^2,
	\end{align*}
	where
	\begin{align*}
		\mathcal M(\tau,\xi) = \int_{\R^2} \tmrho{N_1}(\xi+\eta) \,\tmrho{N_2}(\eta) \,\de(\tau +\dav |\xi+\eta|^2 -\dav |\eta|^2) d\eta .
	\end{align*}
	Therefore, it remains to prove that for all $\tau$ and $\xi$,
	\[
	\mathcal M(\tau,\xi) \les \frac{N_2}{N_1}.
	\]
	Note that
	\[
	\mathcal M(\tau,\xi)  =  \int_{\mathcal C}
	\frac{ \tmrho{N_1}(\xi+\eta) \tmrho{N_2}(\eta)}{|\nabla_\eta (\tau +\dav |\xi|^2 +2\dav \xi\cdot \eta) |}ds_\eta
	\]
	where $\mathcal C = \{\eta: \tau +\dav |\xi|^2 +2\dav \xi\cdot \eta=0\}$, see \cite[Theorem 6.1.5]{Hormander}. Thus,
	\[
	\mathcal M(\tau,\xi)  \les \int_{\mathcal C}\frac{\tmrho{N_1}(\xi+\eta)\tmrho{N_2}(\eta)}{|\xi|}ds_\eta \les \frac{N_2}{N_1}
	\]
	since $|\xi|\gtrsim N_1$ and the curve $\mathcal C \cap \supp(\tmrho{N_2})$ has a length of $\les N_2$.
 \end{proof}

\bigskip


By means of the bilinear Strichartz estimates, we are ready to prove the following bilinear estimates for two functions in appropriate adapted spaces $U_\De^p$ and $V_\De^p$. These bilinear estimates play a crucial role in proving our main theorem. 
\begin{proposition}\label{prop:bilinear-up}
\begin{theoremlist}
\item If $u_1, u_2\in U_\De^{1+2/d}$, then
\beq\label{eq:bilinear_Up2}
\|P_{N_1}u_1 \,P_{N_2}u_2 \|_{L_{t,x}^{1+2/d}} \les \left(\frac{N_{12}^{\min}}{N_{12}^{\max}}\right)^{\beta(d)}  \| P_{N_1}u_1\|_{U_\De^{1+2/d}}\| P_{N_2}u_2\|_{U_\De^{1+2/d}},
\end{equation}
where $\beta(1)=1/6$ and $\beta(2)=1/2$.
\item[\rm(ii-1)] For the one-dimensional case, if $u\in U_\De^{2}$ and $v\in V_\De^{2}$, then
\beq\label{eq:bilinear_V2}
\|P_{N_1}u\, P_{N_2}v \|_{L_{t,x}^{3}} \les \left(\frac{N_{12}^{\min}}{N_{12}^{\max}}\right)^{\frac{1}{6}}  \| P_{N_1}u\|_{U_\De^2}\| P_{N_2}v\|_{V_\De^2}.
\end{equation}
\item[\rm(ii-2)] For the two-dimensional case, if $u\in U_\De^{2}$ and $v\in V_\De^{2}$, then
\beq\label{eq:bilinear_V3}
\|P_{N_1}u \,P_{N_2}v \|_{L_{t,x}^{2}} \les \left(\frac{N_{12}^{\min}}{N_{12}^{\max}}\right)^{\frac{1}{4}}  \| P_{N_1}u\|_{U_\De^2}\| P_{N_2}v\|_{V_\De^2}.
\end{equation}
\end{theoremlist}
\end{proposition}

\begin{proof}

Applying Lemma \ref{lem:transfer} to  \eqref{eq:bilinear1} and \eqref{eq:bilinear} gives the first estimate \eqref{eq:bilinear_Up2}.
Using \eqref{eq:bilinear_Up2} for the one-dimensional case and the embeddings $V_\De^{2} \hookrightarrow U_\De^{3}$, $U_\De^{2} \hookrightarrow U_\De^{3}$, we obtain the second estimate \eqref{eq:bilinear_V2}.

It remains to prove the two dimensional bilinear estimate \eqref{eq:bilinear_V3} for two functions in $U_\De^2$ and $V_\De^2$. First, note that
\beq\label{eq:bilinear_Up}
	\|P_{N_1}u_1 \, P_{N_2}u_2  \|_{L_{t,x}^{2}} \les  \|P_{N_1}u_1\|_{L_{t,x}^{4}}\|P_{N_2}u_2\|_{L_{t,x}^{4}}
	\les \|P_{N_1}u_1\|_{U_\De^{4}}\|P_{N_2}u_2\|_{U_\De^{4}}
\eeq
for all $u_1, u_2\in U_\De^{4}$ which follows from the Cauchy-Schwarz inequality and Lemma \ref{lem:transfer} via the Strichartz estimate \eqref{eq:stri}.
Now let $u \in U_\De^2$ and $v \in V_\De^2$.
On one hand, if $N_1 \sim N_2$, then by \eqref{eq:bilinear_Up}, the embeddings $V_\De^{2} \hookrightarrow U_\De^{4}$ and $U_\De^{2} \hookrightarrow U_\De^{4}$, we see that
\begin{equation*}
		\normo{P_{N_1}u \, P_{N_2}v }_{L_{t,x}^2} \les
  \|P_{N_1}u\|_{U_{\De}^2}\|P_{N_2}v\|_{V_{\De}^2}.
\end{equation*}
On the other hand, if $N_1 \nsim N_2$, without loss of generality we assume that $N_1 \gg N_2$. By applying Lemma  \ref{lem:log-inter} with $p=4$, there exist $\kappa>0$, $w_1 \in \utwo$ and $w_2 \in U_\De^4$ such that $P_{N_2}v = w_1 +w_2$ and
\begin{align}\label{ineq:split}
	\frac{\ka}M \|w_1\|_{\utwo} + e^{M} \|w_2\|_{U_\De^{4}} \leq  \|P_{N_2}v\|_{\vtwo}
\end{align}
for all $M\geq 1$. Since $\supp (\hatt{w}_1)\sim N_2$, it follows from \eqref{eq:bilinear_Up2}, \eqref{eq:bilinear_Up},
and the embedding $U_\De^{2} \hookrightarrow U_\De^{4}$ that  
\begin{align*}
	\|P_{N_1}u\,P_{N_2}v \|_{L_{t,x}^2} &\leq \|(P_{N_1}u)  w_1 \|_{L_{t,x}^2} + \|(P_{N_1}u)  w_2\|_{L_{t,x}^2} \\
	&\les \left(\frac{N_2}{N_1}\right)^{\frac12}\|P_{N_1}u\|_{U_{\De}^2}\|w_1\|_{U_{\De}^2} + \|P_{N_1}u\|_{U_{\De}^2}\|w_2\|_{U_{\De}^4}\\
&= \|P_{N_1}u\|_{U_{\De}^2} \left[ \left(\frac{N_2}{N_1}\right)^{\frac12} \|w_1\|_{U_{\De}^2} + \|w_2\|_{U_{\De}^4} \right].
\end{align*}
If we let $M=-\log \frac{N_2}{N_1}$, then $M\geq 1$ and $ e^{-M} \le \left(\frac{N_2}{N_1}\right)^{1/4}\leq M^{-1}$ since $N_1 \gg N_2$.
Thus, by \eqref{ineq:split} with such an $M$, we have
\[
\left(\frac{N_2}{N_1}\right)^{\frac12} \|w_1\|_{U_{\De}^2} + \|w_2\|_{U_{\De}^4}
\le \left(\frac{N_2}{N_1}\right)^{\frac14} \left( \f{\kappa}{M}  \|w_1\|_{U_{\De}^2} + e^M \|w_2\|_{U_{\De}^4}\right)
\leq \left(\frac{N_2}{N_1}\right)^{\frac14} \|P_{N_2}v\|_{V_{\De}^2}
\]
and therefore
\[
 	\|P_{N_1}u \,P_{N_2}v \|_{L_{t,x}^2} \les \left(\frac{N_2}{N_1}\right)^{\f{1}{4}}\|P_{N_1}u \|_{U_{\De}^2}\|P_{N_2}v \|_{V_{\De}^2}.
\]
\end{proof}

\begin{remark}
   \begin{theoremlist}
   \item Note that the exponent of $N_{12}^{\min}/N_{12}^{\max}$ in each estimate is less than $1$, which we need. Additionally, observe that the one-dimensional bilinear estimates share the common exponent $1/6$, while the two-dimensional bilinear estimates have distinct exponents; $1/2$ and $1/4$. However, the exponent $1/4$ can be replaced by any positive value less than $1/2$.
   \item As seen in the proof, the bilinear estimate for two functions in $U_\De^2$ and $V_\De^2$ was proved through distinct methods in one and two dimensions. If we were to use the bilinear estimate \eqref{eq:bilinear_Up2} in two dimensions, it would prevent the derivation of the $V_\De^2$-norm of $v$ since \eqref{eq:bilinear_Up2} deduces the $U_\De^2$-norm of $v$. This is in contrast to the one-dimensional case where \eqref{eq:bilinear_Up2} deduces $U_\De^3$ and  consequently $V_\De^2$. This is why we employed Lemma \ref{lem:log-inter} for the two-dimensional case.
  \end{theoremlist}
\end{remark}

\section{Proof of Theorem \ref{thm:scattering}}\label{sec:mainproof}
In this section, we prove our main theorem by the contraction mapping argument.
To this end, for any time interval $I\subset \R$, we denote time-restricted space $\{ u= w\big|_{I}: w\in U^2_\De\}$ by $ U^2_\De(I)$ 
and define the Banach space $X_{I}$ as 
\begin{align*}
X_{I}:= \left\{  u \in C(I, L_x^2(\R^d)) \cap U^2_\De(I) : \sum_{N \ge 1} \|P_N u\|_{U_\De^2(I)}^2 < \infty \right\}
\end{align*}
equipped with the norm
\[
\|u\|_{X_{I}} := \left( \sum_{N \ge 1} \| P_N u\|_{U_\De^2(I)}^2  \right)^\frac12.
\]  
We denote by $\mathcal N(u)$ the nonlinear part in \eqref{eq:duhamel}
\begin{align*}
\mathcal N(u)(t) := i \int_0^t \int _0^1 e^{\dav i(t-s)\De}  e^{-ir\Delta} \left(|e^{ir\Delta} u|^{\frac4d} \, e^{ir\Delta} u \right)(s)\,drds
\end{align*}
and obtain the following estimates, first.
\begin{lemma}\label{lem:nonlinear}
For all $u, v\in X_{[0,\infty)}$, 
\beq\label{ineq:normX}
\|\mathcal N(u)\|_{X_{[0,\infty)}}\les \|u\|_{X_{[0,\infty)}}^{1+4/d}
\eeq
and
 \beq\label{ineq:distanceX}
\|\mathcal N(u) -\mathcal N(v) \|_{X_{[0,\infty)}}\les \left(\|u\|_{X_{[0,\infty)}}^{4/d}  + \|v\|_{X_{[0,\infty)}}^{4/d} \right) \|u-v\|_{X_{[0,\infty)}}.
\eeq
\end{lemma}

\begin{proof}
 Let $u\in X_{[0,\infty)}$ and denote its zero extension to $\R$ again by $u$. Then clearly $\| u\|_{X_\R}=\|u\|_{X_{[0,\infty)}}$.
Using the definition of $U_\De^2$ and the fact that $e^{-\dav it\Delta}$ and $P_N$ commute, we have
\beq\label{ineq:Up}
\begin{aligned}
	\|P_N \mathcal N(u)\|_{\utwo} &= \normo{e^{-\dav it\Delta }P_N \mathcal N(u)}_{U^2}\\
	&= \left\| P_N \int_0^t \!\! \int _0^1 e^{-\dav is\De}  e^{-ir\Delta} (|e^{ir\Delta} u|^{\frac4d} e^{ir\Delta} u)(s)\, drds\right\|_{U^2}.
 \end{aligned}
 \eeq
Applying Lemma \ref{lem:duality} to \eqref{ineq:Up} yields
\begin{align*}
\begin{aligned}
\|P_N \mathcal N(u)\|_{\utwo} &=\sup_{\substack{  \|\phi\|_{V^2}=1} } \left|  \int_\R\int _0^1 \left<   P_N e^{-\dav it\De}  e^{-ir\Delta}  (|e^{ir\Delta} u|^{\frac4d} e^{ir\Delta} u)(t) ,\phi(t)  \right>_{L^2_x} dr dt\right|\\
 &=\sup_{\substack{  \|\phi\|_{V^2}=1} } \left|  \int_\R \int _0^1 \left<    (|e^{ir\Delta} u|^{\frac4d} e^{ir\Delta} u)(t) ,   e^{ir\Delta} e^{\dav it\De} P_N\phi(t)  \right>_{L^2_x} dr dt\right|\\
	&= \sup_{\substack{ \|\vp\|_{V_\Delta^2}=1} } \left| \int_\R\int _0^1 \!\!\! \int_{\R^d}  |e^{ir\Delta} u(t,x)|^{\frac4d} e^{ir\Delta} u(t,x)\, \overline {P_N e^{ir\Delta}\vp(t,x)}\, dx drdt\right|,
\end{aligned}
\end{align*}
where we set $ \vp(t,x)= e^{\dav it\De} \phi (t,x)$, then $\vp\in V_\De^2$ with $\|\vp\|_{V_\Delta^2}=1$. Then we have
\begin{align*}
	\|\mathcal N(u)\|_{X_\R}^2
	&=\sum_{N \ge 1}\|P_{N} \mathcal N(u)\|_{\utwo}^2 \\
	&= \sum_{N \ge 1}\sup_{\substack{\|\vp\|_{V_\Delta^2}=1} } \left|  \int _0^1 \!\! \int_\R \int_{\R^d} |e^{ir\Delta} u(t,x
	)|^{\frac4d} e^{ir\Delta} u(t,x)\,  \overline {P_{N}e^{ir\Delta}\vp(t,x)}\, dx dtdr \right|^2\\
 &\leq \sum_{N \ge 1}\sup_{\substack{\|\vp\|_{V_\Delta^2}=1} } \left(  \int _0^1 \left|\int_\R \int_{\R^d} |e^{ir\Delta} u(t,x
	) |^{\frac4d} e^{ir\Delta} u(t,x) \, \overline {P_{N}e^{ir\Delta}\vp(t,x)}\, dx dt \right|dr \right)^2.
 \end{align*}

To obtain \eqref{ineq:normX} and \eqref{ineq:distanceX}, let us consider the two-dimensional case only since the case of $d=1$ is done similarly. The dyadic decomposition  and  the Cauchy-Schwartz inequality lead us to that
 \begin{align}
&	\left|\int_\R\int_{\R^2} e^{ir\Delta} u(t,x)  \overline {e^{ir\Delta} u(t,x)}e^{ir\Delta} u(t,x) \overline {P_{N_4}e^{ir\Delta}\vp(t,x)}\, dx dt \right| \notag \\
& \leq \sum_{N_1, N_2, N_3\geq 1}  \left|\int _\R \int_{\R^2} P_{N_1}e^{ir\Delta} u(t,x)  \overline {P_{N_2}e^{ir\Delta} u(t,x)} P_{N_3}e^{ir\Delta} u(t,x) \overline {P_{N_4}e^{ir\Delta}\vp(t,x)}\, dxdt \right| \notag\\
 &\leq  \sum_{N_1, N_2\geq 1} \normo{ P_{N_1}e^{ir\Delta} u\, P_{N_2}e^{ir\Delta} u }_{L_{t,x}^2} \sum_{N_3\geq1 }\normo{ P_{N_3}e^{ir\Delta} u \, P_{N_4}e^{ir\Delta}\vp }_{L_{t,x}^2}.
 \label{eq:dyadic}
 \end{align}
For the first factor in \eqref{eq:dyadic}, we claim that
\beq \label{eq:claim}
\sum_{N_1, N_2\geq 1} \normo{ P_{N_1}e^{ir\Delta} u\, P_{N_2}e^{ir\Delta} u }_{L_{t,x}^2} \lesssim \|u\|_{X_\R}^2.
\eeq
Indeed, we first note that by Proposition \ref{prop:bilinear-up} (i) and \eqref{eq:persistence-up},
  \beq\label{eq:bilinear_sum_1d}
  \|P_{N_1}e^{ir\Delta}u \, P_{N_2}e^{ir\Delta}u\|_{L_{t,x}^{2}}
   \lesssim  \left(\frac{N_{12}^{\min}}{N_{12}^{\max}}\right)^{\f{1}{2}}  \| P_{N_1}u\|_{U_\De^{2}}\| P_{N_2}u\|_{U_\De^{2}}.
  \eeq
If $N_1 \sim N_2$, then $\frac{N_{12}^{\min}}{N_{12}^{\max}} \sim 1$ and, therefore,  we use the Cauchy-Schwartz inequality to obtain
\[
    	\sum_{N_1\sim N_2} \|P_{N_1}e^{ir\Delta}u\,  P_{N_2}e^{ir\Delta}u \|_{L_{t,x}^{2}}
     \lesssim \sum_{N_1\sim N_2} \| P_{N_1}u\|_{U_\De^{2}}\| P_{N_2}u\|_{U_\De^{2}}
     \les \sum_{N\ge 1}\|P_{N} u\|^2_{U_{\De}^2}=\|u\|_{X_\R} ^2.
\]
On the other hand, if $N_1 \nsim N_2$, we may assume that $N_1 \gg N_2$ without loss of generality.
Then, by \eqref{eq:bilinear_sum_1d}, the Cauchy-Schwartz inequality, and Young's convolution inequality for $\ell^p$-spaces,
we get
\beq
\begin{aligned}\label{eq:young}
&\sum_{N_1 \gg N_2} \|P_{N_1}e^{ir\Delta}u \, P_{N_2}e^{ir\Delta}u \|_{L_{t,x}^{2}}\\
& \lesssim 	\sum_{N_1 \gg N_2} \left(\frac{N_2}{N_1} \right)^\frac12\|P_{N_1}u\|_{U_{\De}^2}\|P_{N_2}u\|_{U_{\De}^2} \\
&\leq  \sum_{N \ge 1} \|P_{N} u\|_{U_\De^2} \left( S_{N} *_\ell \|P_{N}u\|_{U_\De^2}\right)_{N}  \\
&\leq  \Bigg(\sum_{N \ge 1}\|P_{N} u\|^2_{U_{\De}^2}\Bigg)^{\f{1}{2}} \Bigg( \sum_{N \ge 1} \left| \left( S_{N} *_\ell \|P_{N}u\|_{U_\De^2}\right)_{N} \right| ^2 \Bigg)^{\f{1}{2}}\\
&\les  \sum_{N \ge 1}\|P_{N} u\|^2_{U_{\De}^2}= \|u\|_{X_\R}^2,
\end{aligned}
\eeq
which proves the claim \eqref{eq:claim}.
Here, we used the notation that the convolution of $\{a_N\}_{N\geq1}$ and $\{b_N\}_{N\geq 1}$ is given by
\begin{align*}
	(a_N *_\ell b_N)_N := \sum_{\wt{N}\geq 1} a_{N/\wt{N}} b_{\wt{N}}
\end{align*}
and
\begin{align*}
	S_N = \begin{cases}
		N^{-\frac12}  & \mbox{ for } N \ge 1\\
		0       & \mbox{ for } 0 < N <1.
	\end{cases}
\end{align*}
Hence, \eqref{eq:dyadic} and \eqref{eq:claim} yield
\begin{equation}\label{eq:a1}
\begin{aligned}
\|\mathcal N(u)\|_{X_\R}^2
&\les \sum_{N_4 \ge 1}\sup_{\|\vp\|_{V_\Delta^2}=1 }\left(\int_0^1 \|u\|_{X_\R}^2
  \sum_{N_3\geq1 }\normo{ P_{N_3}e^{ir\Delta} u \, P_{N_4}e^{ir\Delta}\vp }_{L_{t,x}^3}   dr \right)^2\\
  & = \|u\|_{X_\R}^4\sum_{N_4 \ge 1}\sup_{\|\vp\|_{V_\Delta^2}=1 }\left( \int_0^1
  \sum_{N_3\geq1 }\normo{ P_{N_3}e^{ir\Delta} u \, P_{N_4}e^{ir\Delta}\vp }_{L_{t,x}^2}   dr\right)^2.
	\end{aligned}
\end{equation}
Now, by applying Proposition \ref{prop:bilinear-up} (ii-2) to \eqref{eq:a1},
we have
\begin{align}
	\|\mathcal N(u)\|_{X_\R}^2
	&\les \|u\|_{X_\R}^4  \sum_{N_4 \ge 1} \sup_{\|\vp\|_{V_\Delta^2}=1 } \left( \int_0^1 \sum_{N_3 \ge 1 }   \left(\frac{N_{34}^{\min}}{N_{34}^{\max}} \right)^\frac14 \| P_{N_3}e^{ir\Delta} u \|_{U_{\De}^2}  \| P_{N_4}e^{ir\Delta}\vp \|_{V_{\De}^2}    dr \right)^2 \notag \\
	& \leq \|u\|_{X_\R}^4 \sum_{N_4 \ge 1} \left( \sum_{N_3 \ge 1 }   \left(\frac{N_{34}^{\min}}{N_{34}^{\max}} \right)^\frac14 \| P_{N_3} u \|_{U_{\De}^2}   dr \right)^2,\label{ineq:N-1}
\end{align}
where we used   \eqref{eq:persistence-vp}, \eqref{eq:persistence-up} and the fact that $\|P_N \vp \|_{V_{\De}^2}  \leq \|\vp \|_{V_{\De}^2}$.
Again, as in the argument of \eqref{eq:young},
using Young's convolution inequality, one sees
\begin{align*}
	\sum_{N_4 \ge 1} \left( \sum_{N_3 \ge 1 }   \left(\frac{N_{34}^{\min}}{N_{34}^{\max}} \right)^\frac14 \| P_{N_3} u \|_{U_{\De}^2}   dr \right)^2 \les \|u\|_{X_\R}^2,
\end{align*}
which together with \eqref{ineq:N-1} proves \eqref{ineq:normX}.

In order to prove the second bound \eqref{ineq:distanceX},  if we define
\[
\mathcal N(u_1, u_2, u_3) := i\int_0^t \int _0^1 e^{\dav i(t-s)\De}  e^{-ir\Delta} \left(e^{ir\Delta} u_1 \overline{e^{ir\Delta} u_2} e^{ir\Delta} u_3 \right)(s)\,drds,
\]
then the same argument as \eqref{ineq:normX} together with Proposition \ref{prop:bilinear-up} gives
\[
\| \mathcal N(u_1, u_2, u_3) \|_{X_\R}\lesssim \|u_1\|_{X_\R}\|u_2\|_{X_\R}\|u_3\|_{X_\R}.
\]
Then we complete the proof, once we observe that
\[
\mathcal N(u) -\mathcal N(v)= \mathcal N(u-v, u,u) + \mathcal N(v, u-v, u)+\mathcal N(v,v,u-v).
\]
\end{proof}

\begin{proof}[Proof of Theorem \ref{thm:scattering}]
 Due to the time reversal symmetry, we consider the forward case $t \in [0,\infty)$ only.
For each $\de>0$, we consider the complete metric space $(B_\de,d)$ defined by
\begin{align*}
	B_\de:= \left\{  u \in  C([0,\infty), L_x^2(\R^d)) \cap U^2_\De ([0,\infty)): \|u\|_{X_{[0,\infty)}} \le \de \right\},
\end{align*}
equipped with the distance
\begin{align*}
	d(u,v) : = \|u-v\|_{X_{[0,\infty)}}.
\end{align*}
Define the map $\Phi$ on $(B_\de,d)$ by
\begin{align*}
	\Phi(u) = e^{\dav it\Delta}u_0 + i\int_0^t \int _0^1 e^{\dav i(t-s)\De}  e^{-ir\Delta} (|e^{ir\Delta} u|^{\frac4d} \, e^{ir\Delta} u)(s) drds.
\end{align*}
Its linear part can be bounded as
\begin{equation*}
\begin{aligned}
    \|e^{\dav it\Delta} u_0 \|_{X_{[0,\infty)}}^2 &= \sum_{N\ge1} \|P_N e^{\dav it\Delta}u_0\|_{\utwo ([0,\infty))}^2\\
    &= \sum_{N\ge1} \|P_N u_0\|_{U^2 ([0,\infty))}^2 \\
    &= \sum_{N\ge1} \|P_N u_0\|_{L_x^2}^2 \approx \|u_0\|_{L_x^2}^2,
 \end{aligned}
\end{equation*}
since any normalized function in $L_x^2(\R^d)$ is a $U^2$-atom.
Thus, by Lemma \ref{lem:nonlinear}, there exists a positive constant $C$ such that
\beq\label{eq:cont1}
\|\Phi(u)\|_{X_{[0,\infty)}} \le C \left(\|u_0\|_{L_x^2} + \de^{1+4/d} \right)
\eeq
and
\beq\label{eq:cont2}
d(\Phi(u)- \Phi(v))\le C \de^{4/d} d(u,v)
\eeq
for all $u,v \in B_\de$.
If $\de>0$ is chosen to be small enough so that $2C\de^{4/d}<1$ and the initial data $u_0$ satisfies $2C\|u_0\|_{L_x^2} \le \de$, then it follows from \eqref{eq:cont1} and \eqref{eq:cont2} that $\Phi$ is a self-contraction on $(B_\de,d)$.
Thus, $\Phi$ has a fixed point $u$ which is the unique global solution to \eqref{eq:main} in $ B_\de$.

Now, we show the scattering result in $L_x^2(\R^d)$. For the initial data $u_0\in L_x^2(\R^d)$ with $\|u_0\|_{L_x^2}\les \delta$, the global solution $u$ satisfies
\[
u(t) = e^{\dav it \Delta}\left(u_0 + ie^{-\dav it \Delta}  \mathcal N (u) (t)\right).
\]
Moreover, it follows from \eqref{ineq:normX} and the embedding $U_\De^{2} \hookrightarrow V_\De^{2}$ that
\begin{align*}
	\sum_{N \ge 1} \|P_N \mathcal N (u)\|_{\vtwo}^2 \les \sum_{N \ge 1} \|P_N \mathcal N (u)\|_{\utwo}^2 \les 1.
\end{align*}
By applying Lemma \ref{lem:limit}, there exists a limit
\begin{align*}
	u_0 +   e^{-\dav it \Delta}\sum_{N \ge 1} P_N \mathcal N(u)(t)\to u_+
\end{align*}
in $L_x^2(\R^d)$ as $t \to \infty$, which implies the following scattering phenomenon
\begin{align*}
	\|u(t) - e^{\dav it\Delta} u_+\|_{L_x^2} \xrightarrow{t \to \infty} 0.
\end{align*}
This completes the proof of Theorem \ref{thm:scattering}.
\end{proof}

 \begin{remark}\label{rem:general-regul}
	Regarding the scattering phenomenon on $H^s(\R^d)$ for any $s > 0$, we can apply a similar approach as described above. More precisely, we replace $X_{[0,\infty)}$ by the Banach space $X^s_{[0,\infty)}$ defined by
	\begin{align*}
		X^s_{[0,\infty)}:= \Bigg\{  u \in C([0,\infty), H^s(\R^d)) \cap U^2_\De ([0,\infty))  : \|u\|_{X^s_{[0,\infty)}} < \infty \Bigg\}
	\end{align*}
 equipped with the norm
\[
\|u\|_{X^s_{[0,\infty)}} := \Bigg(\sum_{N \ge 1} N^{2s}\|P_N u\|_{U_\De^2([0,\infty))}^2 \Bigg)^\frac12, 
\]  
and obtain improved bilinear estimates. This leads to a small data global well-posedness and scattering on $H^s(\R^d)$ for any $s>0$ via the contraction mapping argument.
	\end{remark}


	\noindent
	\textbf{Acknowledgements:}
	The authors are supported by the National Research Foundation of Korea (NRF) grants funded by the Korean government (MSIT) NRF-2019R1A5A1028324, (MSIT) NRF-2020R1A2C1A01010735, (MOE) NRF-2021R1I1A1A01045900 and (MOE) NRF-2022R1I1A1A01056408.
	
	\renewcommand{\thesection}{\arabic{chapter}.\arabic{section}}
	\renewcommand{\theequation}{\arabic{chapter}.\arabic{section}.\arabic{equation}}
	\renewcommand{\thetheorem}{\arabic{chapter}.\arabic{section}.\arabic{theorem}}

	\bibliographystyle{abbrv}

	\def\cprime{$'$}
	\small{

\end{document}